\documentclass[graybox]{svmult}

\usepackage{amsfonts,amstext,amssymb,physics} % must be loaded first here before other packages
\usepackage{amsmath}

\usepackage{amssymb}
\usepackage{helvet}         % selects Helvetica as sans-serif font
\usepackage{courier}        % selects Courier as typewriter font
\usepackage{type1cm}        % activate if the above 3 fonts are
\usepackage{graphicx}        % standard LaTeX graphics tool
\usepackage{multicol}        % used for the two-column index
\usepackage[bottom]{footmisc}% places footnotes at page bottom

\usepackage{enumerate}

\usepackage{enumitem}  % http://ctan.org/pkg/enumitem
\usepackage{calc}% http://ctan.org/pkg/calc
\setlist{labelindent=1pt,itemsep=0.1cm}
\setlist[itemize]{leftmargin=0.7cm}
\setlist[enumerate]{itemindent=0em,leftmargin=0.7cm}

\usepackage{float} % for placement of tables and figures %

\usepackage[bookmarks=false,hidelinks,unicode]{hyperref}
\hypersetup{hypertexnames=false}  %removes some special hypertext warnings

\smartqed

\allowdisplaybreaks

\begin{document}
\title*{Unique common fixed point results for four weakly commuting maps in $G$-metric spaces}
\titlerunning{Unique common fixed point results for weakly commuting maps} %for an abbreviated version of your contribution title if the original one is too long
\author{Talat Nazir \and Sergei Silvestrov}
\authorrunning{T. Nazir, S. Silvestrov} %for an abbreviated version of your contribution title if the original one is too long

\institute{Talat Nazir,
\at
Department of Mathematical Sciences, University of South Africa, Florida 0003, South Africa. \\
\email{talatn@unisa.ac.za}
\and
Sergei Silvestrov,
\at Division of Mathematics and Physics, School of Education, Culture and Communication, M{\"a}lardalen University, Box 883, 72123 V{\"a}ster{\aa}s, Sweden. \\  \email{sergei.silvestrov@mdu.se}
\and }

% \index[aut]{Silvestrov, Sergei}

\maketitle
%\label{chap:NazirSilvestrov:UCFPGMS}

\abstract*{Using the setting of $G$-metric spaces, common fixed point theorems for four maps satisfying the weakly commuting conditions are obtained for various generalized contractive conditions. Several examples are also presented to show the validity of main results.
\keywords{weakly commuting maps, common fixed point, generalized contraction, $G$-metric space}\\
{\bf MSC 2020 Classification:} 54H25, 47H10, 54E50}

\abstract{Using the setting of $G$-metric spaces, common fixed point theorems for four maps satisfying the weakly commuting conditions are obtained for various generalized contractive conditions. Several examples are also presented to show the validity of main results.
\keywords{weakly commuting maps, common fixed point, generalized contraction, $G$-metric space}\\
{\bf MSC 2020 Classification:} 54H25, 47H10, 54E50}

\section{Introduction}
\label{secNaSe:IntroductionG-metricspace}
The study of unique common fixed points of mappings satisfying certain
contractive conditions has been at the center of strong research activity.
Mustafa and Sims \cite{ZM01} generalized the concept of a metric space.
Based on the notion of generalized metric spaces, Mustafa {\it et al.}
\cite{ZM02, ZM03, ZM04, ZM05, ZMe} obtained some fixed point theorems for
mappings satisfying different contractive conditions. Study of common fixed
point theorems in generalized metric spaces was initiated by Abbas and
Rhoades \cite{Abbas}. Also, Abbas {\it et al.} \cite{Abbas2} obtained some
periodic point results in generalized metric spaces. Saadati {\it et al.}
\cite{Saadati} studied some fixed point results for contractive mappings in
partially ordered $G$-metric spaces. Shatanawi \cite{Wasfi} obtained fixed
points of $\Phi $-maps in $G$-metric spaces. Also, Shatanawi \cite{Wasfi2}
obtained a coupled coincidence fixed point theorem in the setting of a
generalized metric spaces for two mappings $F$ and $g$ under certain
conditions with an assumption of $G$-metric spaces continuity of one of the mapping
involved therein. While, Chugh {\it et al.} \cite{Chugh} obtained some fixed
point results for maps satisfying property $p$ in $G$-metric space. Several other useful results in $G$-metric space appeared in
\cite{Abbasnz,Abbas13,KhanAbbas,Latif2018,LatifNazir} and also in the book \cite{AgarvalKarapinarOReganRLHierroBook2018FPTMetrTypeSp} and reference therein.

The aim of this work is to initiate the study of common fixed
point for four maps satisfying weakly commuting conditions in $G$-metric space.

Throughout this work, $\mathbb{R},$ $\mathbb{R}_{\geq 0}$ and $\mathbb{Z}_{\geq 0}$ will denote the set of all real numbers, the set of all nonnegative real
numbers and the set of all non-negative integers, respectively. Consistent with Mustafa and Sims \cite{ZM02,ZM03,ZM04}, the following
definitions and results will be needed in the sequel.

\begin{definition}
	\label{DefNaSeG:1.1} Let $X$ be a nonempty set. Suppose that
	$G:X\times X\times X\rightarrow \mathbb{R}_{\geq 0}$ is a mapping satisfying
	\begin{enumerate}[leftmargin=1cm]
		\item[{\rm (${G}_{1}$)}] \hspace{0,1cm} $G(x,y,z)=0$ if $x=y=z;$
		\item[{\rm (${G}_{2}$)}] \hspace{0,1cm} $0<G(x,y,z)$ for all $x,y\in X,$ with $x\neq y;$
		\item[{\rm (${G}_{3}$)}] \hspace{0,1cm} $G(x,x,y)\leq G(x,y,z)$ for all $x,y,z\in X,$ with
        $y\neq z;$
		\item[{\rm (${G}_{4}$)}] \hspace{0,1cm} $G(x,y,z)=G(x,z,y)=G(y,z,x)=\ldots,$
        (symmetry in all variables); 		
		\item[{\rm (${G}_{5}$)}] \hspace{0,1cm} $G(x,y,z)\leq G(x,a,a)+G(a,y,z)$ for all $x,y,z,a\in
		X. $
\end{enumerate}
	Then $G$ is called a $G$-metric on $X$ and $(X,G)$ is called a
$G$-metric space.
\end{definition}
\begin{example} {\rm (\cite{ZM02})}
	\label{ExamNaSeGM:1.2}
If $X$ is a non-empty subset and $d:X\times X\rightarrow
\mathbb{R}_{\geq 0}$ is a  metric on $X$, then the functions $G_{1},G_{2}:X\times
X\times X\rightarrow
\mathbb{R}_{\geq 0}$, given by
\begin{align*}
& G_{1}\left( x,y,z\right) = d\left( x,y\right) +d\left( y,z\right) +d\left(z,x\right), \\
& G_{2}\left( x,y,z\right) = \max \{d\left( x,y\right) ,d\left(y,z\right),d\left(z,x\right) \}
\end{align*}
are $G$-metrics on $X$.
\end{example}
\begin{example}{\rm (\cite{AgarvalKarapinarOReganRLHierroBook2018FPTMetrTypeSp})} \label{ExamNaSeGM:1.3}  Every non-empty set $X$ can be provided with the discrete $G$-metric, which is defined, for all $x,y,z\in X$, by 
\[
G(x,y,z)=\left \{
\begin{array}{cc}
	0 & \text{if }x=y=z \\
	1 & \text{otherwise.}
\end{array}
\right.
\]
\end{example}

\begin{example} {\rm (\cite{AgarvalKarapinarOReganRLHierroBook2018FPTMetrTypeSp})}
\label{ExamNaSeGM:1.4}
Let $X=\mathbb{R}_{\geq 0}$ be the interval of nonnegative real numbers. Then the function $G:X\times X\times X\rightarrow
\mathbb{R}_{\geq 0}$, given by
\[
G(x,y,z)=\left \{
\begin{array}{cc}
	0 & \text{if }x=y=z \\
	\max \{x,y,z\} & \text{otherwise.}%
\end{array}\right.
\]
is a $G$-metric on $X$.
	\end{example}

\begin{proposition}
For a given $G$-metric space $(X,G)$,
\begin{enumerate}[leftmargin=1cm]
	\item[$($E$_{1})$] \hspace{0,3cm} $G_{1}(r,u,s)=\kappa G(r,u,s)\}$ for some $\kappa >0,$
    \item[$($E$_{2})$] \hspace{0,3cm} $G_{1}(r,u,s)=\min \{ \kappa ,G(r,u,s)\}$ for some $\kappa >0,$		
    \item[$($E$_{3})$] \hspace{0,3cm} $G_{2}(r,u,s)=\dfrac{G(r,u,s)}{1+G(r,u,s)}$
	\end{enumerate}
are also $G$-metrics on $X.$ Further, if $X=\bigcup\limits_{i=1}^{n}A_{i}$ is any
	partition of $X$ and $\kappa >0,$ then
\begin{enumerate}[leftmargin=1cm]
		\item[$($E$_{4})$] \hspace{0,3cm} $G_{3}(r,u,s)=\left \{
		\begin{array}{l}
			G(r,u,s),\text{ if for some }i\text{ we have }r,u,s\in A_{i} \\
			\kappa +G(r,u,s),\text{ otherwise}%
		\end{array}		\right. $
\end{enumerate}
is a $G$-metric on $X$.
\end{proposition}

We can drive the following useful properties form a given $G$-metric on $X$.

\begin{proposition}[\cite{ZM04}]
	\label{GProp1}  For a given $G$-metric space $(X,G)$, the following properties hold for any
$r,u,s,y,x\in X$:
	\begin{enumerate}[label=\textup{(\roman*)}, leftmargin=1cm, ref=(\roman*)]
		\item \hspace{0,3cm} if $G(r,u,s)=0,$ then $r=s=u,$
		\item \hspace{0,3cm} $G(r,y,s)\leq G(r,r,y)+G(y,s,y),$
		\item \hspace{0,3cm} $G(r,r,u)\leq 2G(u,r,u),$
		\item \hspace{0,3cm} $G(r,u,s)\leq G(r,x,s)+G(x,u,s),$
		\item \hspace{0,3cm} $G(r,u,s)\leq \dfrac{2}{3}(G(r,u,y)+G(r,y,s)+G(y,u,s)),$
		\item \hspace{0,3cm} $G(r,u,s)\leq G(r,y,y)+G(y,u,y)+G(x,x,s),$
		\item \hspace{0,3cm} $\left \vert G(r,u,s)-G(r,u,y)\right \vert \leq \max \{ G(y,s,s),G(s,y,y)\},$
		\item \hspace{0,3cm} $\left \vert G(r,u,s)-G(r,u,x)\right \vert \leq G(r,x,s),$
		\item \hspace{0,3cm} $\left \vert G(r,u,s)-G(u,s,s)\right \vert \leq \max \{ G(r,s,s),G(s,r,r)\},$
		\item \hspace{0,3cm} $\left \vert G(r,y,y)-G(y,r,r)\right \vert \leq \max \{ G(y,r,r),G(r,y,y)\}.$
	\end{enumerate}
\end{proposition}

\begin{definition}[\cite{ZM03}]
	\label{DefNaSeG:1.2}  A sequence $\{x_{n}\}$ in a $G$-metric
	space $X$ is	
	\begin{enumerate}[label=\textup{(\roman*)}, ref=(\roman*)]
		\item a $G$-{\it Cauchy} sequence if, for any $\varepsilon >0,$ there
		is a $n_{0} \in \mathbb{Z}_{\geq 0}$ such that for all $%
		n,m,l\geq n_{0},$ $G(x_{n},x_{m},x_{l})<\varepsilon ,$
		\item a $G$-{\it convergent} sequence if, for any $\varepsilon >0,$
		there is an $x\in X$ and a $n_{0}\in \mathbb{Z}_{\geq 0}$ such that for all $n,m\geq n_{0},$
		$G(x,x_{n},x_{m})<\varepsilon $.
	\end{enumerate}
\end{definition}
\begin{definition}[\cite{ZM04}]
	\label{DefNaSeG:1.3}  A $G$-metric space on $X$ is said to be $G$-complete if every
$G$-Cauchy sequence in $X$ is $G$-convergent in $X.$ It is known that $\{x_{n}\}$ $G$-converges to $x\in X$ if and only if $G(x_{m},x_{n},x)%
	\rightarrow 0$ as $n,m\rightarrow \infty .$
\end{definition}

\begin{proposition}[\cite{ZM02}] \label{PropNaSeG:1.3}  Let $X$ be a $G$-metric
	space. Then the following properties are equivalent:
	
	\begin{enumerate} [label=\textup{\arabic*)}, ref=\arabic*]
		\item $\{x_{n}\}$ is $G$-convergent to $x$.		
		\item $G(x_{n},x_{n},x)\rightarrow 0$ as $n\rightarrow \infty .$		
		\item $G(x_{n},x,x)\rightarrow 0$ as $n\rightarrow \infty .$		
		\item $G(x_{n},x_{m},x)\rightarrow 0$ as $n,m\rightarrow \infty .$
	\end{enumerate}
\end{proposition}

\begin{definition}[\cite{ZM02}]
	\label{DefNaSeG:1.4}  A $G$-metric on $X$ is said to be
	symmetric if $G(x,y,y)=G(y,x,x)$ for all $x,y\in X.$
\end{definition}

\begin{proposition}[\cite{ZM01}]
	\label{PropNaSeG:1.5}  Every $G$-metric on $X$ will define a
	metric $d_{G}$ on $X$ by%
	\begin{equation}
		d_{G}(x,y)=G(x,y,y)+G(y,x,x),\text{ }\forall \text{ }x,y\in X.  \tag{1.1}
	\end{equation}%
	For a symmetric $G$-metric%
	\begin{equation}
		d_{G}(x,y)=2G(x,y,y),\text{ }\forall \text{ }x,y\in X.  \tag{1.2}
	\end{equation}%
	However, if $G$ is non-symmetric, then the following inequality holds:%
	\begin{equation}
		\frac{3}{2}G(x,y,y)\leq d_{G}(x,y)\leq 3G(x,y,y),\text{ }\forall \text{ }%
		x,y\in X.  \tag{1.3}
	\end{equation}%
	It is also obvious that
	\[
	G(x,x,y)\leq 2G(x,y,y).
	\]
\end{proposition}
Now, we give an example of a non-symmetric $G$-metric.

\begin{example}
	If for some real numbers $a ,b$ and $c$ we have $X=\{a,b,c\}$, the mapping $G:X\times X\times X\rightarrow \mathbb{R}_{\geq 0}$ be given as%
	\begin{equation*}
		\begin{tabular}{|c|c|}
			\hline
			$(a,b,c)$ & $G(a,b,c)$ \\ \hline \hline
			$(a ,a , a ),(b ,b ,b ),(c ,c ,c )$
			& $0$ \\ \hline
			$(a ,b ,b ),(b ,a ,b ),(b ,b ,a )$ &
			$1$ \\ \hline
			$%
			\begin{array}{c}
				(a ,a ,b ),(a ,b ,a ),(b ,a ,a ),
				\\
				(b ,c ,c ),(c ,b ,c ),(c ,c ,b )
			\end{array}%
			$ & $2$ \\ \hline
			$%
			\begin{array}{c}
				(a ,a ,c ),(a ,c ,a ),(c ,a ,a), \\
				(a ,c ,c ),(c ,a ,c ),(c ,c ,a)
			\end{array}%
			$ & $3$ \\ \hline
			$%
			\begin{array}{c}
				(b ,b ,c ),(b ,c ,b ),(c ,b ,b ), \\
				(a , b ,c ),(a ,c ,b ),(b ,a ,c ),
				\\
				(b ,c ,a ),(c ,a ,b ),(c ,b ,a )
			\end{array}
			$ & $4$ \\ \hline
		\end{tabular}
	\end{equation*}
	is a non-symmetric $G$-metric on $X$, since $G(a ,a ,c)\neq G(a ,c ,c )$.
\end{example}

\begin{example}
	If $X$ is a non-empty subset and $d:X\times X\rightarrow
	\mathbb{R}_{\geq 0}$ is metric on $X$. Then for some value of $\kappa >0,$ the
	function $G:X\times X\times X\rightarrow
	\mathbb{R}_{\geq 0}$ given by%
	\begin{eqnarray*}
		G\left( x,x,y\right)  &=&G\left( x,y,x\right) =G\left( y,x,x\right) :=\kappa
		d\left( x,y\right),  \\
		G\left( x,y,y\right)  &=&G\left( y,x,y\right) =G\left( y,y,x\right) :=2\kappa
		d\left( x,y\right)
	\end{eqnarray*}
	is a non-symmetric $G$-metric on $X$.
	\end{example}

\begin{definition}
Let $f$ and $g$ be self mappings on a $G$-metric space $X$. The mappings $f$
and $g$ are called commuting if $gfu = fgu$, for all $u \in X$, (\cite{Abbas}).
\end{definition}
In 1982, Sessa \cite{Sessa} introduced the concept of weakly commuting maps
in metric spaces as follows.

\begin{definition}
	\label{DefinNaSeG:1.7}
	Let $(X,d)$ be a metric space and $f$
	and $g$ be two self mappings of $X$. Then $f$ and $g$ are called weakly
	commuting if
	\[
	d(fgx,gfx)\leq d(fx,gx)
	\]
	holds for all $x\in X$.
\end{definition}

The concept of weakly commuting maps in $G$-metric space is defined as:
\begin{definition}
	\label{DefinNaSeG:1.8}
	Let $(X,G)$ be a $G$-metric space and $%
	f $ and $g$ be two self mappings of $X$. Then $f$ and $g$ are called weakly
	commuting if
	\[
	G(fgx,gfx,gfx)\leq G(fx,gx,gx)
	\]
	holds for all $x\in X$.
\end{definition}

\section{\bf Common Fixed Point Theorems in G-metric Spaces}

In this section, we obtain common fixed point theorems for four
mappings that satisfying a set of conditions in $G$-metric space. We start
with the following result.
\begin{theorem}
	\label{ThmNaSeG:2.1}
	Let $X$ be a complete $G$-metric space, and
	$A,B,S,T:X\rightarrow X$ be mappings satisfying
	\begin{eqnarray}
	G(Sx,Ty,Ty) &\leq & h\max \{G(Ax,By,By),G(Sx,Ax,Ax),G(Ty,By,By)\},\quad  \label{tGeq2.1} \\
	G(Sx,Sx,Ty) &\leq & h\max \{G(Ax,Ax,By),G(Sx,Sx,Ax),G(Ty,Ty,By)\}. \quad  \label{tGeq2.2}
	\end{eqnarray}
	Assume the maps $A,B,S$ and $T$ satisfy the following conditions:	
	\begin{enumerate} [label=\textup{\arabic*)}, ref=\arabic*]
	\item$TX\subseteq AX$ and $SX\subseteq BX$,
	\item The mappings $A$ and $B$ are sequentially continuous, 		
	\item The pairs $\{A,S\}$ and $\{B,T\}$ are weakly commuting.
	\end{enumerate}
	If $h\in \lbrack 0,1)$, then $A,B,S$ and $T$ have a unique common fixed
	point.
\end{theorem}
\begin{proof} \smartqed Suppose that $f$ and $g$ satisfy the conditions
	\eqref{tGeq2.1}
	) and \eqref{tGeq2.2}. If $G$ is symmetric, then by adding above two inequalities,
	we have
	\begin{equation}
		d_{G}(Sx,Ty)\leq h\max \{d_{G}(Ax,By),d_{G}(Sx,Ax),d_{G}(Ty,By)\}
		\label{tGeq2.3}
	\end{equation}
	for all $x,y\in X$ with $0\leq h<1$, and the fixed point of $A,B,S$ and $T$
	follows from \cite{Sessa}.
	
	Now if $X$ is non-symmetric $G$-metric space. Then by the
	definition of metric $(X,d_{G})$ and \eqref{tGeq2.3}, we obtain
	\begin{align*}
		&d_{G}(Sx,Ty) =G(Sx,Ty,Ty)+G(Sx,Sx,Ty) \\
		&\leq h\max \{[G(Ax,By,By)+G(Ax,Ax,By)],[G(Sx,Ax,Ax) \\
		&+G(Sx,Sx,Ax)],[G(Ty,By,By)+G(Ty,Ty,By)]\} \\
		&\leq h\max \{[\frac{2}{3}d_{G}(Ax,By)+\frac{2}{3}d_{G}(Ax,By)],[\frac{2}{3}
		d_{G}(Ax,Sx) \\
		&+\frac{2}{3}d_{G}(Sx,Ax)],[\frac{2}{3}d_{G}(Ty,By)+\frac{2}{3}
		d_{G}(Ty,By)]\} \\
		&=\frac{4h}{3}\max \{[d_{G}(Ax,By),d_{G}(Sx,Ax),d_{G}(Ty,By)]\}
	\end{align*}
	for all $x\in X$. Here, the contractivity factor $\dfrac{4h}{3}$ may not be
	less than $1$. Therefore metric gives no information. In this case, for
	given a $x_{0}\in X$, choose $x_{1}\in X$ such that $Tx_{0}=Ax_{1}$, choose
$x_{2}\in X$ such that $Sx_{1}=Bx_{2}$, choose $x_{3}\in X$ such that
$Ax_{3}=Tx_{2}$. Continuing as above process, we can construct a sequence
$\{x_{n}\}$ in $X$ such that $Ax_{2n+1}=Tx_{2n}$, $n\in \mathbb{Z}_{\geq 0}$
	and $Bx_{2n+2}=Sx_{2n+1}$, $n\in \mathbb{Z}_{\geq 0}$. Let
	\begin{align}
	y_{2n} & = Ax_{2n+1}=Tx_{2n},\, \, \,n\in \mathbb{Z}_{\geq 0},
	\\ 	
    y_{2n+1} &= Bx_{2n+2}=Sx_{2n+1},\, \, \,n\in \mathbb{Z}_{\geq 0}.
	\end{align}
	Given $n>0$. If $n$ is even, then $n=2k$ for some $k\geq 1$.
	Then from \eqref{tGeq2.1}, we have
	\begin{align*}
		&G(y_{n},y_{n+1},y_{n+1})=G(y_{2k},y_{2k+1},y_{2k+1})\\
        &=G(Tx_{2k},Sx_{2k+1},Sx_{2k+1})=G(Sx_{2k+1},Sx_{2k+1},Tx_{2k}) \\
		&\leq h \max\{G(Ax_{2k+1},Ax_{2k+1},Bx_{2k}),G(Sx_{2k+1},Sx_{2k+1},Ax_{2k+1}), \\
		&\hspace{7cm} G(Tx_{2k},Tx_{2k},Bx_{2k})\} \\
		&=h\max\{G(y_{2k},y_{2k},y_{2k-1}),G(y_{2k+1},y_{2k+1},y_{2k}),G(y_{2k},y_{2k},y_{2k-1})\}
		\\
		&=h\max \{G(y_{2k},y_{2k},y_{2k-1}),G(y_{2k+1},y_{2k+1},y_{2k})\} \\
		&=h\max \{G(y_{n},y_{n},y_{n-1}),G(y_{n+1},y_{n+1},y_{n})\},
	\end{align*}
	which implies that
	\begin{equation}
		G(y_{n},y_{n+1},y_{n+1})\leq hG(y_{n-1},y_{n},y_{n}). \label{tGeq2.4}
	\end{equation}
	If $n$ is odd, then $n=2k+1$ for some $k\in \mathbb{Z}_{\geq 0}$. Again, from \eqref{tGeq2.1}, we
	have
	\begin{align*}
		&G(y_{n},y_{n+1},y_{n+1}) =G(y_{2k+1},y_{2k+2},y_{2k+2}) \\
        &=G(Tx_{2k+1},Sx_{2k+2},Sx_{2k+2})=G(Sx_{2k+2},Sx_{2k+2},Tx_{2k+1}) \\
		&\leq h\max\{G(Ax_{2k+2},Ax_{2k+2},Bx_{2k+1}),G(Sx_{2k+2},Sx_{2k+2},Ax_{2k+2}), \\
		&\hspace{6cm} G(Tx_{2k+1},Tx_{2k+1},Bx_{2k+1})\} \\
		&=h\max
		\{G(y_{2k+1},y_{2k+1},y_{2k}),G(y_{2k+2},y_{2k+2},y_{2k+1}),G(y_{2k+1},y_{2k+1},y_{2k})\}
		\\
		&=h\max \{G(y_{2k+1},y_{2k+1},y_{2k}),G(y_{2k+2},y_{2k+2},y_{2k+1})\} \\
		&=h\max \{G(y_{n},y_{n},y_{n-1}),G(y_{n+1},y_{n+1},y_{n})\},
	\end{align*}
	that is,
	\[
	G(y_{n},y_{n+1},y_{n+1})\leq hG(y_{n-1},y_{n},y_{n}).
	\]
	Thus, for each $n\in \mathbb{Z}_{> 0}$, we have
	\begin{equation}
		G(y_{n},y_{n+1},y_{n+1})\leq h^{n}G(y_{0},y_{1},y_{1}). \label{tGeq2.5}	
	\end{equation}%
	Thus, if $y_{0}=y_{1}$, we get $G(y_{n},y_{n+1},y_{n+1})=0$ for each $n\in
	\mathbb{Z}_{\geq 0}$. Hence $y_{n}=y_{0}$ for each $n\in \mathbb{Z}_{\geq 0}$. Therefore $\{y_{n}\}$
	is $G$-Cauchy. So we may assume that $y_{0}\neq y_{1}$. Let $n,$ $m\in \mathbb{Z}_{\geq 0}$ with $m>n$. By axiom $G_{5}$ of the definition of $G$ metric space, we
	have
	\[
	G(y_{n},y_{m},y_{m})\leq  \sum\limits_{i=n}^{m-1} G(y_{k},y_{k+1},y_{k+1}).
	\]
	By   \eqref{tGeq2.5}, we get
	\begin{equation*}
		G(y_{n},y_{m},y_{m}) \leq
		h^{n}\sum_{i=0}^{m-1-n}h^{i}G(y_{0},y_{1},y_{1}) \leq \frac{h^{n}}{1-h}G(y_{0},y_{1},y_{1}).
	\end{equation*}
	Taking the limit when $m,n\rightarrow \infty $ yields
	\[
	\lim\limits_{m,n\rightarrow \infty }G(y_{n},y_{m},y_{m})=0.
	\]
	So we conclude that $\{y_{n}\}$ is a Cauchy sequence in $X$. Since $X$ is
$G$-complete, then it yields that $\{y_{n}\}$ and hence any subsequence of $\{y_{n}\}$ converges to some $z\in X$. So that, the subsequences
$\{Ax_{2n+1}\}$, $\{Bx_{2n+2}\}$, $\{Sx_{2n+1}\}$ and $\{Tx_{2n}\}$ converge
	to $z$. First suppose that $A$ is sequentially continuous. So,
	\[\lim\limits_{n\rightarrow \infty }A\left( Ax_{2n+1}\right) =Az, \quad
	\lim\limits_{n\rightarrow \infty }A\left( Sx_{2n+1}\right) =Az.
	\]
	Since $\{A,S\}$ is weakly commuting, we have
	\[
	G(SAx_{2n+1},ASx_{2n+1},ASx_{2n+1})\leq G(Sx_{2n+1},Ax_{2n+1},Ax_{2n+1}).
	\]
	On taking limit as $n\rightarrow \infty $, we get that
$G(SAx_{2n+1},Az,Az) \rightarrow 0$. Thus,
	\[
	\lim\limits_{n\rightarrow \infty }SAx_{2n+1}=Az.
	\]
	Assume $Az\neq z$, we get
	\begin{multline*}
 G(SAx_{2n+1},Tx_{2n},Tx_{2n})  \\
  \leq h\max \{G(AAx_{2n+1},Bx_{2n},Bx_{2n}), G(SAx_{2n+1},AAx_{2n+1},AAx_{2n+1}),\\
 G(Tx_{2n},Bx_{2n},Bx_{2n})\}.
	\end{multline*}
	On letting limit as $n\rightarrow \infty $, we have
	\[
	G(Az,z,z)\leq h\max \{G(Az,z,z),G(Az,z,z),G(z,z,z)\}.
	\]
	So, we conclude that
	\[
	G(Az,z,z)\leq hG(Az,z,z),
	\]
	which is a contradiction. So $Az=z$. Also,
	\begin{align*}
		& G(Sz,Tx_{2n},Tx_{2n})
		\leq h\max
		\{G(Az,Bx_{2n},Bx_{2n}),G(Sz,Az,Az),\\
&\hspace{7cm} G(Tx_{2n},Bx_{2n},Bx_{2n})\}.
	\end{align*}
	By taking limit as $n\rightarrow \infty $,
	\begin{equation*}
		G(Sz,z,z) \leq h\max \{G(Az,z,z),G(Sz,Az,Az),G(z,z,z)\} =hG(Sz,z,z).
	\end{equation*}
	We get $G(Sz,z,z)=0$. So $Sz=z$. Suppose $B$ is sequentially continuous, then
	\[
	\lim\limits_{n\rightarrow \infty }B(Bx_{2n})=Bz, \quad \lim\limits_{n\rightarrow \infty
	}B(Tx_{2n})=Bz.
	\]
	Since the pair $\{B,T\}$ is weakly commuting, we have
	\[
	G(TBx_{2n},BTx_{2n},BTx_{2n})\leq G(Tx_{2n},Bx_{2n},Bx_{2n}).
	\]
	Taking the limit as $n\rightarrow \infty $, we get $G(TBx_{2n},Bz,Bz)%
	\rightarrow 0$. Thus
	\[
	\lim\limits_{n\rightarrow \infty }T(Bx_{2n})=Bz.
	\]
	Assume $Bz\neq z$. Since
	\begin{align*}
		&G(Sx_{2n+1},TBx_{2n},TBx_{2n}) \\
		&\leq h\max
		\{G(Ax_{2n+1},BBx_{2n},BBx_{2n}),G(Sx_{2n+1},Ax_{2n+1},Ax_{2n+1}), \\
		&\hspace{7cm} G(TBx_{2n},BBx_{2n},BBx_{2n})\}.
	\end{align*}
	Again taking the limit when $n\rightarrow \infty $, implies
	\begin{align*}
		G(z,Bz,Bz) &\leq h\max \{G(z,Bz,Bz),G(z,z,z),G(Bz,Bz,Bz)\}\\
		&=hG(z,Bz,Bz),
	\end{align*}
	a contradiction. Hence $Bz=z$. Since
	\begin{align*}
		G(Sx_{2n+1},Sx_{2n+1},Tz) &\leq h\max \{G(Ax_{2n+1},Ax_{2n+1},Bz), \\
		&\hspace{1cm} G(Sx_{2n+1},Sx_{2n+1},Ax_{2n+1}),G(Tz,Tz,Bz)\}.
	\end{align*}
	Taking the limit when $n\rightarrow \infty $ yields
	\begin{align*}
		G(z,z,Tz) &\leq h\max \{G(z,z,Bz),G(z,z,z),G(Tz,Tz,Bz)\}\\
&=hG(z,Tz,Tz).
	\end{align*}
	Also,
	\begin{eqnarray*}
		G(Sx_{2n+1},Tz,Tz) &\leq &h\max \{G(Ax_{2n+1},Bz,Bz), \\
		&&G(Sx_{2n+1},Ax_{2n+1},Ax_{2n+1}),G(Tz,Bz,Bz)\}.
	\end{eqnarray*}
Taking the limit when $n\rightarrow \infty $, we obtain
	\begin{eqnarray*}
		G(z,Tz,Tz) &\leq &h\max \{G(z,Bz,Bz),G(z,z,z),G(Tz,Bz,Bz)\} \\
		&=&hG(z,z,Tz).
	\end{eqnarray*}
	Thus, $G(z,z,Tz)\leq hG(z,Tz,Tz)\leq h^{2}G(z,z,Tz).$
	Since $h<1$, we get $G(z,z,Tz)=0$. Hence $Tz=z$. So, $z$ is a common fixed
	point for $A,B,S$ and $T$. Now, we shall prove that $z$ is a unique common
	fixed point. Let $w$ be a common fixed point for $A,B,S$ and $T$ with $w\neq
	z$. Then
	\begin{eqnarray*}
		G(z,w,w) &=&G(Sz,Tw,Tw) \\
		&\leq &h\max \{G(Az,Bw,Bw),G(Sz,Az,Az),G(Tw,Bw,Bw)\} \\
		&=&h\max \{G(z,w,w),G(z,z,z),G(w,w,w)\}=hG(z,w,w),
	\end{eqnarray*}
	which is a contradiction. So $z=w$.
\qed \end{proof}
In Theorem \ref{ThmNaSeG:2.1}, if we take $S:X\rightarrow X$ by $Sx=Tx$, then obtain
the following corollary.

\begin{corollary}
	\label{CorolNaSeG:1.7} Let $X$ be a complete $G$-metric space,
	and let $A,B,S:X\rightarrow X$ be mappings satisfying%
	\begin{align*}
		G(Sx,Sy,Sy) &\leq h\max \{G(Ax,By,By),G(Sx,Ax,Ax),G(Sy,By,By)\}, % \label{tGeq2.6}	
\\
		G(Sx,Sx,Sy) &\leq h\max \{G(Ax,Ax,By),G(Sx,Sx,Ax),G(Sy,Sy,By)\}. % \label{tGeq2.7}
	\end{align*}
	Assume the maps $A,B$ and $S$ satisfy the following conditions:
	
	\begin{enumerate} [label=\textup{\arabic*)}, ref=\arabic*]
	\item $SX\subseteq AX\cup BX$,
	\item The mappings $A$ and $B$ are sequentially continuous,
	\item The pairs $\{A,S\}$ and $\{B,S\}$ are weakly commuting.
	\end{enumerate}
\end{corollary}
If $h\in \lbrack 0,1)$, then $A,B$ and $S$ have a unique common fixed point.
\begin{corollary}
	\label{CorolNaSeG:2.3}
	Let $X$ be a complete $G$-metric space,
	and $A,S,T:X\rightarrow X$ be mappings satisfying
	\begin{align*}
		G(Sx,Ty,Ty)\leq h\max \{G(Ax,Ay,Ay),G(Sx,Ax,Ax),G(Ty,Ay,Ay)\},  % \label{tGeq2.8}
	\\
		G(Sx,Sx,Ty)\leq h\max \{G(Ax,Ax,Ay),G(Sx,Sx,Ax),G(Ty,Ty,Ay)\}.  % \label{tGeq2.9}
	\end{align*}
	Assume the maps $A,$ $S$ and $T$ satisfy the following conditions:
		\begin{enumerate} [label=\textup{\arabic*)}, ref=\arabic*]
		\item $TX\cup SX\subseteq AX$,
		\item The mappings $A$ is sequentially continuous,
        \item The pairs $\{A,S\}$ and $\{A,T\}$ are weakly commuting.
	\end{enumerate}
If $h\in \lbrack 0,1)$, then $A,$ $S$ and $T$ have a unique common fixed
point.
\end{corollary}

If we take $S=T$ in Corollary \ref{CorolNaSeG:2.3}, then we get the following Corollary that extend and generalized Theorem 4.3.1 of \cite{AgarvalKarapinarOReganRLHierroBook2018FPTMetrTypeSp}.

\begin{corollary}
	\label{CorolNaSeG:2.3b}
	Let $X$ be a complete $G$-metric space,
	and let $A,S:X\rightarrow X$ be mappings satisfying%
	\begin{align*}
		G(Sx,Sy,Sy)\leq h\max \{G(Ax,Ay,Ay),G(Sx,Ax,Ax),G(Sy,Ay,Ay)\},  % \label{tGeq2.8}
		\\
		G(Sx,Sx,Sy)\leq h\max \{G(Ax,Ax,Ay),G(Sx,Sx,Ax),G(Sy,Sy,Ay)\}.  % \label{tGeq2.9}
	\end{align*}
\noindent Assume the maps $A$ and $S$ satisfy the following conditions:
	\begin{enumerate} [label=\textup{\arabic*)}, ref=\arabic*]
		\item $ SX\subseteq AX$,
		\item The mappings $A$ is sequentially continuous,
		\item The pairs $\{A,S\}$ is weakly commuting.
	\end{enumerate}
If $h\in \lbrack 0,1)$, then $A$ and $S$ have a unique common fixed
	point.
\end{corollary}

\begin{corollary}
	\label{CorolNaSeG:2.4}
	Let $X$ be a complete $G$-metric space,
	and let $S,T:X\rightarrow X$ be mappings satisfying%
	\begin{align*}
		G(Sx,Ty,Ty)&\leq h\max \{G(x,y,y),G(Sx,x,x),G(Ty,y,y)\}, %\label{tGeq2.10}
	\\
		G(Sx,Sx,Ty)&\leq h\max \{G(x,x,y),G(Sx,Sx,x),G(Ty,Ty,y)\}.  %\label{tGeq2.11}
	\end{align*}%
\end{corollary}
If $h\in \lbrack 0,1)$, then $S$ and $T$ have a unique common fixed point.
\begin{corollary}
	\label{CorolNaSeG:2.5}
	Let $X$ be a complete $G$-metric space,
	and let $S:X\rightarrow X$ be a\ mapping satisfying
	\begin{align*}
		G(Sx,Sy,Sy)&\leq h\max \{G(x,y,y),G(Sx,x,x),G(Sy,y,y)\},  %\label{tGeq2.12}
\\
		G(Sx,Sx,Sy)&\leq h\max \{G(x,x,y),G(Sx,Sx,x),G(Sy,Sy,y)\}.  %\label{tGeq2.13}
	\end{align*}
	If $h\in \lbrack 0,1)$, then $S$ has a unique fixed point.
\end{corollary}
Now, we introduce an example of Theorem \ref{ThmNaSeG:2.1}.
\begin{example}
	\label{ExampleNaSeG:2.6}
	Let $X=[0,1]$, Define $A,B,S,T:X\rightarrow
	X$ by $Ax=\frac{x}{3}$, $Bx=\frac{x}{6}$, $Sx=\frac{x}{9}$, and $Tx=\frac{x}{12}$. Then clearly $TX\subseteq AX$, $SX\subseteq BX$. Also, note that the
	pairs $\{A,S\}$ and $\{B,T\}$ are weakly compatible.
    Define the mapping $G:X\times
	X\times X\rightarrow \mathbb{R}_{\geq 0}$ by 	
    \[
	G(x,y,z)=\max \{|x-y|,|x-z|,|y-z|\}.
	\]
	Then, $(X,G)$ is a complete $G$-metric. Also,
	\begin{align*}
	& G(Sx,Ty,Ty)=\frac{1}{36}|x-y|,&& G(Ax,By,By)=\frac{1}{6}|x-y|,
	\\
	& G(Sx,Ax,Ax)=\frac{2x}{9}, && G(Ty,By,By)=\frac{y}{12}.
	\end{align*}
	Now
	\begin{eqnarray*}
		G(Sx,Ty,Ty) &=&\frac{1}{36}|x-y|=\frac{1}{3}\left( \frac{1}{12}|x-y|\right)
		=\frac{1}{3}\max \{ \frac{1}{6}|x-y|,\frac{2x}{9},\frac{y}{12}\} \\
		&=&h\max \{G(Ax,By,By),G(Sx,Ax,Ax),G(Ty,By,By)\}.
	\end{eqnarray*}
	Note that $A,B,S$ and $T$ satisfy the hypothesis of Theorem \ref{ThmNaSeG:2.1}. Here,
	$0$ is the unique common fixed point of $A,B,S$ and $T$.
\end{example}

\begin{theorem}
	\label{TheoremNaSeG:2.7}
	Let $X$ be a complete $G$-metric space,
	and $A,B,S,T:X\rightarrow X$ be mappings satisfying
	\begin{align}
	&	G(Sx,Ty,Ty)\leq \kappa (G(Sx,Ax,Ax)+G(Ty,By,By)), \label{tGeq2.14}
\\ 		
& G(Sx,Sx,Ty)\leq \kappa(G(Sx,Sx,Ax)+G(Ty,Ty,By)).  \label{tGeq2.15}
	\end{align}
	Assume the maps $A,B,S$ and $T$ satisfy the following conditions:
		\begin{enumerate} [label=\textup{\arabic*)}, ref=\arabic*]
		\item $TX\subseteq AX$ and $SX\subseteq BX$,
		\item The mappings $A$ and $B$ are sequentially continuous, 		
		\item The pairs $\{A,S\}$ and $\{B,T\}$ are weakly commuting.
	\end{enumerate}
	
	If $\kappa \in \lbrack 0,\frac{1}{2})$, then $A,B,S$ and $T$ have a unique common
	fixed point.
\end{theorem}
\begin{proof} \smartqed If $X$ is a symmetric $G$-metric space, then by
	adding above two inequalities we obtain
	\begin{align*}
		&G(Sx,Ty,Ty)+G(Sx,Sx,Ty) \\
		&\leq  \kappa (G(Sx,Ax,Ax)+G(Sx,Sx,Ax)+G(Ty,By,By)+G(Ty,Ty,By),
	\end{align*}
	which further implies that
	\[
	d_{G}(Sx,Ty)\leq \kappa(d_{G}(Sx,Ax)+d_{G}(Ty,By)),
	\]
	for all $x,y\in X$ with $0\leq \kappa<\frac{1}{2}$, and the fixed point of $A,B,S$
	and $T$ follows from \cite{Sessa}.
	
	Now if $X$ is non-symmetric $G$-metric space. Then by the
	definition of metric $(X,d_{G})$, we obtain
	\begin{align*}
		& d_{G}(Sx,Ty) =G(Sx,Ty,Ty)+G(Sx,Sx,Ty) \\
		&\leq \kappa(G(Sx,Ax,Ax)+G(Sx,Sx,Ax)+G(Ty,By,By)+G(Ty,Ty,By)) \\
		&\leq \kappa( \frac{2}{3}d_{G}(Ax,Sx)+\frac{2}{3}d_{G}(Sx,Ax)+\frac{2}{3}%
		d_{G}(Ty,By)+\frac{2}{3}d_{G}(Ty,By)) \\
		&=\frac{4\kappa}{3}(d_{G}(Sx,Ax)+d_{G}(Ty,By))
	\end{align*}
	for all $x\in X$. Here, the contractivity factor $\dfrac{4\kappa}{3}$ may not be
	less than $\dfrac{1}{2}$. Therefore, metric gives no information. In this
	case, for given a $x_{0}\in X$, choose $x_{1}\in X$ such that $Ax_{1}=Tx_{0}$, choose $x_{2}\in X$ such that $Sx_{1}=Bx_{2}$, choose $x_{3}\in X$ such
	that $Ax_{3}=Tx_{2}$. Continuing as above process, we can construct a
	sequence $\{x_{n}\}$ in $X$ such that $Ax_{2n+1}=Tx_{2n}$, $n\in \mathbb{Z}_{\geq 0}$ and $Bx_{2n+2}=Sx_{2n+1}$, $n\in \mathbb{Z}_{\geq 0}$. Let
	\begin{align*}
	& y_{2n}=Ax_{2n+1}=Tx_{2n},\, \, \,n\in \mathbb{Z}_{\geq 0},
	\\
	& y_{2n+1}=Bx_{2n+2}=Sx_{2n+1},\, \, \,n\in \mathbb{Z}_{\geq 0}.
	\end{align*}
	Given $n>0$. If $n$ is even, then $n=2k$ for some $k> 0$.
	Then from \eqref{tGeq2.14}, we have
	\begin{align*}
		&G(y_{n},y_{n+1},y_{n+1})=G(y_{2k},y_{2k+1},y_{2k+1}) \\
        &=G(Tx_{2k},Sx_{2k+1},Sx_{2k+1})=G(Sx_{2k+1},Sx_{2k+1},Tx_{2k}) \\
		&\leq \kappa(G(Sx_{2k+1},Sx_{2k+1},Ax_{2k+1})+G(Tx_{2k},Tx_{2k},Bx_{2k}))\\
		&=\kappa(G(y_{2k+1},y_{2k+1},y_{2k})+G(y_{2k},y_{2k},y_{2k-1})) \\
		&=\kappa(G(y_{n+1},y_{n+1},y_{n})+G(y_{n},y_{n},y_{n-1})),
	\end{align*}
	which further implies that
	\[
	G(y_{n},y_{n+1},y_{n+1})\leq \frac{\kappa}{1-\kappa}G(y_{n-1},y_{n},y_{n}),
	\]
	or $G(y_{n},y_{n+1},y_{n+1})\leq \lambda _{1}G(y_{n-1},y_{n},y_{n}),$ where
    $\lambda _{1}=\dfrac{\kappa}{1-\kappa}<1.$ If $n$ is odd, then $n=2k+1$ for some $k\geq 0$. Again, from \eqref{tGeq2.15},
	\begin{align*}
		& G(y_{n},y_{n},y_{n+1}) =G(y_{2k+1},y_{2k+1},y_{2k+2})=G(Sx_{2k+1},Sx_{2k+1},Tx_{2k+2}) \\
		&\leq \kappa(G(Sx_{2k+1},Sx_{2k+1},Ax_{2k+1})+G(Tx_{2k+2},Tx_{2k+2},Bx_{2t+2}))
		\\
		&=\kappa(G(y_{2k+1},y_{2k+1},y_{2k})+G(y_{2k+2},y_{2k+2},y_{2k+1})) \\
		&=\kappa(G(y_{n},y_{n},y_{n-1})+G(y_{n+1},y_{n+1},y_{n})),
	\end{align*}
	that is
	\[
	G(y_{n},y_{n+1},y_{n+1})\leq \frac{\kappa}{1-\kappa}G(y_{n-1},y_{n},y_{n}),
	\]
	or $G(y_{n},y_{n+1},y_{n+1})\leq \lambda _{2}G(y_{n-1},y_{n},y_{n}),$ where $%
	\lambda _{2}=\dfrac{\kappa}{1-\kappa}<1.$ Thus, with
    $\lambda =\max \{ \lambda _{1},\lambda_{2}\}$, we have for each $n> 0$,
	\begin{equation}
		G(y_{n},y_{n+1},y_{n+1})\leq \lambda ^{n}G(y_{0},y_{1},y_{1}).   \label{tGeq2.16}
	\end{equation}%
	Thus, if $y_{0}=y_{1}$, we get $G(y_{n},y_{n+1},y_{n+1})=0$ for each $n\in
	\mathbb{Z}_{\geq 0}$. Hence $y_{n}=y_{0}$ for each $n\in \mathbb{Z}_{\geq 0}$. Therefore $\{y_{n}\}$
	is $G$-Cauchy. So we may assume that $y_{0}\neq y_{1}$. Let $n,$ $m\in \mathbb{Z}_{\geq 0}$ with $m>n$. By axiom $G_{5}$ of the definition of $G$ metric space, we
	have
	\[
	G(y_{n},y_{m},y_{m})\leq  \sum\limits_{j=n}^{m-1}G(y_{k},y_{k+1},y_{k+1}).
	\]
	By \eqref{tGeq2.16}, we get
	\begin{eqnarray*}
		G(y_{n},y_{m},y_{m}) \leq  \lambda ^{n}\sum_{i=0}^{m-1-n}\lambda ^{i}G(y_{0},y_{1},y_{1}) \leq \frac{\lambda ^{n}}{1-\lambda }G(y_{0},y_{1},y_{1}).
	\end{eqnarray*}
	Taking limit when $m,n\rightarrow \infty $, we have
	\[
	\lim\limits_{m,n\rightarrow \infty }G(y_{n},y_{m},y_{m})=0.
	\]
	So we conclude that $\{y_{n}\}$ is a Cauchy sequence in $X$. Since $X$ is
$G$-complete, then it yields that $\{y_{n}\}$ and hence any subsequence of
$\{y_{n}\}$ converges to some $z\in X$. So that, the subsequences
$\{Ax_{2n+1}\}$, $\{Bx_{2n+2}\}$, $\{Sx_{2n+1}\}$ and $\{Tx_{2n}\}$ converge
	to $z$. First suppose that $A$ is sequentially continuous. Then,
	\[
	\lim\limits_{n\rightarrow \infty }A(Ax_{2n+1})=Az, \quad \lim\limits_{n\rightarrow
		\infty }A\left( Sx_{2n+1}\right) =Az.
	\]
	Since $\{A,S\}$ is weakly commuting, we have
	\[
	G(SAx_{2n+1},ASx_{2n+1},ASx_{2n+1})\leq G(Sx_{2n+1},Ax_{2n+1},Ax_{2n+1}).
	\]
	On taking limit as $n\rightarrow \infty $, we get that $G(SAx_{2n+1},Az,Az)%
	\rightarrow 0$. Thus,
	\[
	\lim\limits_{n\rightarrow \infty }SAx_{2n+1}=Az.
	\]
	Assume $Az\neq z$, we get
	\begin{multline*}
	G(SAx_{2n+1},Tx_{2n},Tx_{2n}) \\
\leq \kappa(G(SAx_{2n+1},AAx_{2n+1},AAx_{2n+1})+G(Tx_{2n},Bx_{2n},Bx_{2n})).
	\end{multline*}
	On letting limit as $n\rightarrow \infty $, we have
	\[
	G(Az,z,z)\leq \kappa(G(Az,z,z)+G(z,z,z)).
	\]
	Since $k<1$, we conclude that
	$
	G(Az,z,z)<G(Az,z,z),
	$
	which is a contradiction. So, $Az=z$. Also,
	$$
	G(Sz,Tx_{2n},Tx_{2n})\leq \kappa(G(Sz,Az,Az)+G(Tx_{2n},Bx_{2n},Bx_{2n})).
	$$
	Taking the limit when $n\rightarrow \infty $ yields
	\[
		G(Sz,z,z) \leq \kappa(G(Sz,Az,Az)+G(z,z,z)) =\kappa G(Sz,z,z).
	\]
	Since $k<1$, we get $G(Sz,z,z)=0$. So $Sz=z$. Suppose $B$ is sequentially
	continuous. Then,
	$
	\lim\limits_{n\rightarrow \infty }B(Bx_{2n})=Bz$ and  $\lim\limits_{n\rightarrow \infty}B(Tx_{2n})=Bz.
	$
	Since the pair $\{B,T\}$ is weakly commuting, we have
	\[
	G(TBx_{2n},BTx_{2n},BTx_{2n})\leq G(Tx_{2n},Bx_{2n},Bx_{2n}).
	\]
	Taking the limit as $n\rightarrow \infty $, we get
$G(TBx_{2n},Bz,Bz)\rightarrow 0$. Thus,  	
\[
\lim\limits_{n\rightarrow \infty }T(Bx_{2n})=Bz.
\]
	Assume $Bz\neq z$. Since
	\begin{multline*}
		G(Sx_{2n+1},TBx_{2n},TBx_{2n}) \\
        \leq \kappa(G(Sx_{2n+1},Ax_{2n+1},Ax_{2n+1})+G(TBx_{2n},BBx_{2n},BBx_{2n})).
	\end{multline*}
	Again taking the limit as $n\rightarrow \infty $, implies
	\[ 		G(z,Bz,Bz) \leq \kappa(G(z,z,z)+G(Bz,Bz,Bz)) < G(z,Bz,Bz),
	\]
	a contradiction. Hence $Bz=z$. Since
	\begin{eqnarray*}
		G(Sx_{2n+1},Sx_{2n+1},Tz) \leq \kappa (G(Sx_{2n+1},Sx_{2n+1},Ax_{2n+1})+G(Tz,Tz,Bz)),
	\end{eqnarray*}
	taking the limit $n\rightarrow \infty $ yields
	\begin{eqnarray*}
		G(z,z,Tz) \leq \kappa (G(z,z,z)+G(Tz,Tz,Bz) ) = 2\kappa G(z,z,Tz).
	\end{eqnarray*}
	Since $2 \kappa <1$, we get $G(z,z,Tz)=0$. Hence $Tz=z$. So $z$ is a common fixed
	point for $A,B,S$ and $T$. Let us prove that $z$ is a unique common fixed point.
	Assume that $w$ be a common fixed point for $A,B,S$ and $T$ with $w\neq z$. Then
	\begin{eqnarray*}
		G(z,w,w) &=& G(Sz,Tw,Tw) \leq \kappa (G(Sz,Az,Az)+G(Tw,Bw,Bw)) \\
		&=& \kappa (G(z,z,z)+G(w,w,w)) =\kappa G(z,w,w) <G(z,w,w),
	\end{eqnarray*}
	which is a contradiction. So $z=w$.
\qed \end{proof}

Now, if we take $S:X\rightarrow X$ by $Sx=Tx$ in Theorem \ref{TheoremNaSeG:2.7}, then obtain the following Corollary.
\begin{corollary}
	\label{CorolNaSeG:2.8}
	Let $X$ be a complete $G$-metric space,
	and let $A,B,S:X\rightarrow X$ be mappings satisfying
	\begin{align*} % \label{tGeq2.17}
		& G(Sx,Sy,Sy)\leq \kappa(G(Sx,Ax,Ax)+G(Sy,By,By)), \\
		& G(Sx,Sx,Sy)\leq \kappa(G(Sx,Sx,Ax)+G(Sy,Sy,By)).
% \label{tGeq1.8}
	\end{align*}
	Assume the maps $A,B$ and $S$ satisfy the following conditions:
		\begin{enumerate} [label=\textup{\arabic*)}, ref=\arabic*]
		\item $SX\subseteq AX\cup BX$,
		\item The mappings $A$ and $B$ are sequentially continuous, 		
        \item The pairs $\{A,S\}$ and $\{B,S\}$ are weakly commuting.
	\end{enumerate}
If $\kappa \in \lbrack 0,\frac{1}{2})$, then $A,B$ and $S$ have a unique common fixed point.
\end{corollary}

\begin{corollary}
	\label{CorolNaSeG:2.9}
	Let $X$ be a complete $G$-metric space,
	and let $A,T:X\rightarrow X$ be mappings satisfying
	\begin{align*}
	G(Tx,Ty,Ty)\leq \kappa (G(Tx,Ax,Ax)+G(Ty,Ay,Ay)),  % \label{tGeq2.19}
    \\ 		
    G(Tx,Tx,Ty)\leq \kappa (G(Tx,Tx,Ax)+G(Ty,Ty,Ay)).  % \label{tGeq2.20}
	\end{align*}
Assume the maps $A$ and $T$ satisfy the following conditions:
		\begin{enumerate} [label=\textup{\arabic*)}, ref=\arabic*]
		\item $TX\subseteq AX$,		
		\item The mappings $A$ is sequentially continuous,	
		\item The pairs $\{A,T\}$ are weakly commuting.
	\end{enumerate}
	If $\kappa \in \lbrack 0,\frac{1}{2})$, then $A$ and $T$ have a unique common
	fixed point.
\end{corollary}
\begin{proof} \smartqed Define $A:X\rightarrow X$ by $Ax=Bx$ and
	define $T:X\rightarrow X$ by $Tx=Sx$. Then the four maps $A,B,S$ and $T$
	satisfy all the hypothesis of Theorem \ref{TheoremNaSeG:2.7}. So the result follows from Theorem \ref{TheoremNaSeG:2.7}.
\qed \end{proof}

\begin{corollary}
	\label{CorolNaSeG:2.10}
	Let $X$ be a complete $G$-metric space,
	and let $S:X\rightarrow X$ be a mapping satisfying
	\begin{align*}
	G(Sx,Sy,Sy) &\leq \kappa(G(Sx,x,x)+G(Sy,y,y)),   % \label{tGeq2.21}
	\\ 		
    G(Sx,Sx,Sy) &\leq \kappa(G(Sx,Sx,x)+G(Sy,Sy,y)).  % \label{tGeq2.22}
	\end{align*}
	If $\kappa\in \lbrack 0,\frac{1}{2})$, then $S$ has a unique fixed point.
\end{corollary}

\section*{Conclusion}
\label{secNaS1:Conclusion}
In this work, we obtained several common fixed point results for self-mappings that are satisfying generalized contractive conditions under the weakly commuting conditions in the frame work of $G$-metric spaces. We also established some examples for obtained results. These results extend, unify and generalize several results in the existing literature.

%%\input{referenc}
%
%\printindex
%%\addcontentsline{toc}{chapter}{Index}
%\printindex[aut]
%\addcontentsline{toc}{chapter}{Author Index}

\end{document}